\documentclass[12pt, reqno]{article}
\title{Random walks on stochastic hyperbolic half planar triangulations}
\author{Omer Angel \and Asaf Nachmias \and Gourab Ray}
\date{\small \today}

\usepackage{amsmath,amssymb,amsfonts,amsthm}

\usepackage{hyperref}
\hypersetup{
    colorlinks=true,
    linkcolor=blue,
    citecolor=red,
    urlcolor=blue,
    pdfborder={0 0 0}
}            

\usepackage{graphicx}
\usepackage{xcolor}

\usepackage[font=sf, labelfont={sf,bf}, margin=1cm]{caption}

\usepackage{cleveref}
  \crefname{theorem}{Theorem}{Theorems}
  \crefname{thm}{Theorem}{Theorems}
  \crefname{lemma}{Lemma}{Lemmas}
  \crefname{lem}{Lemma}{Lemmas}
  \crefname{remark}{Remark}{Remarks}
  \crefname{prop}{Proposition}{Propositions}
  \crefname{defn}{Definition}{Definitions}
  \crefname{corollary}{Corollary}{Corollaries}
  \crefname{section}{Section}{Sections}
  \crefname{figure}{Figure}{Figures}

\newtheorem{thm}{Theorem}[section]
\newtheorem{lemma}[thm]{Lemma}
\newtheorem{corollary}[thm]{Corollary}
\newtheorem{prop}[thm]{Proposition}
\newtheorem{defn}[thm]{Definition}

\numberwithin{equation}{section}

\renewcommand{\P}{\mathbb P}
\newcommand{\Z}{\mathbb Z}
\newcommand{\E}{\mathbb E}
\newcommand{\R}{\mathbb R}

\newcommand{\eps}{\varepsilon}

\usepackage{bbm}
\newcommand{\wH}{\widetilde{\mathbb H}}
\renewcommand{\H}{\mathbb H}
\newcommand{\F}{\mathbb F}
\newcommand{\bP}{\mathbf{P}}
\newcommand{\cE}{\mathcal{E}}
\newcommand{\cP}{\mathcal{P}}

\newcommand{\lf}{\lfloor}
\newcommand{\rf}{\rfloor}
\newcommand{\A}{\mathcal{A}}
\newcommand{\Cut}{\operatorname{Cut}}

\begin{document}
\maketitle
\abstract{We study the simple random walk on stochastic hyperbolic half planar triangulations constructed in \cite{AR13}. We show that almost surely the walker escapes the boundary of the map in positive speed and that the return probability to the starting point after $n$ steps scales like $\exp(-cn^{1/3})$.}

\section{Introduction}\label{sec:intro}

In this paper we study the behavior of the simple random walk on random half planar hyperbolic triangulations. The latter are probability measures on rooted half planar maps satisfying two natural properties, translation invariance and the domain Markov property. In \cite{AR13} these measures were constructed and characterized as a one parameter family $\H_\alpha$ where $\alpha\in[0,1)$ is the probability that the face containing the root edge has an internal vertex. See further definitions below.

In \cite{Ray13} it is proved that the geometry of the map exhibits a phase transition at the value $\alpha=2/3$. When $\alpha<2/3$ the map almost surely has quadratic volume growth, infinitely many cut-sets of bounded size and the random walker is in the ``Alexander-Orbach'' regime, that is, after $t$ steps it is typically at distance $t^{1/3}$ (the same behavior as in random trees \cite{kestensubdiff, BK06} and in high-dimensional critical percolation \cite{GN09}). When $\alpha>2/3$ the map is almost surely ``hyperbolic'' in the sense that it has exponential volume growth and positive anchored expansion. The hyperbolic regime is the focus of the current paper.

Since these maps are not sufficiently regular (namely, they are not transitive and have unbounded degrees) one cannot apply standard tools (such as \cite{B00}) to study basic properties the random walk such as its speed and return probabilities. Indeed, a special treatment, which employs the inherent randomness of the map, is needed. In this paper we show that in the hyperbolic phase $\alpha\in (2/3,1)$ the distance of the random walker from the boundary grows linearly (in particular, its speed is positive) and that the return probabilities follow a stretched exponential law $\exp(-cn^{1/3})$.

\begin{thm}\label{T:speed}
  Fix $\alpha\in(2/3,1)$ and let $H$ be a random half planar triangulation
  with law $\H_\alpha$ with boundary $\partial H$. Consider the simple random walk
  $X_n$ on $H$ and write $d^H(x,y)$ for the graph distance between $x$ and
  $y$ in $H$. Then almost surely we have
  \begin{equation}\label{posspeed}
    \liminf_n \frac{d^H(X_n,\partial H)}{n} > 0.\nonumber
  \end{equation}
\end{thm}

\begin{thm}\label{T:return_prob}
  Fix $\alpha \in (2/3,1)$. Given the map $H$ with law $\H_\alpha$, let $\bP_H$ denote the law of a simple
  random walk on $H$ starting from $\rho$. There are positive constants $c,C$ depending only on
  $\alpha$ so that $\H_\alpha$-a.s.\ for all large enough $n$
  \begin{equation}\label{returnprob}
    e^{-C n^{1/3}} \leq
    \bP_H(X_n = \rho)
    \leq e^{-c n^{1/3}}.\nonumber
  \end{equation}
\end{thm}

\subsection{Half planar maps}

Recall that a \textbf{planar map} is a proper embedding (that is, with no crossing edges) of a connected (multi) graph on the sphere viewed up to orientation preserving homeomorphisms from the sphere to itself.
Connected components of the complement of the embedding are called \textbf{faces}.
We shall focus maps with a \textbf{boundary}, that is one face is marked as the external face and the edges and vertices incident to it form the boundary of the map. Vertices that are not on the boundary are called \textbf{internal vertices}. In this paper the boundary will always be simple, that is, the boundary edges and vertices form a simple cycle or a bi-infinite simple path. A \textbf{triangulation} is a map where every face has precisely three edges except possibly an external face. If the external face of a triangulation is a simple cycle with $p$-edges, we say it is a \textbf{triangulation of a $p$-gon.} \textbf{Half planar triangulations} are triangulations which are locally finite, one-ended (that is, the removal of any finite set of vertices results in precisely one infinite cluster) and have a bi-infinite simple boundary. In other words, these triangulations can be embedded such that the union of all vertices, edges and faces which are not the external face equals $\R \times \R_+$.
All our maps are \textbf{rooted}, that is, an oriented edge is specified as the root. In a half planar map, the root is always on the boundary and is oriented in a way such that the external face is to the right of the root.

\begin{figure}[t]
\centering{\includegraphics[scale=1]{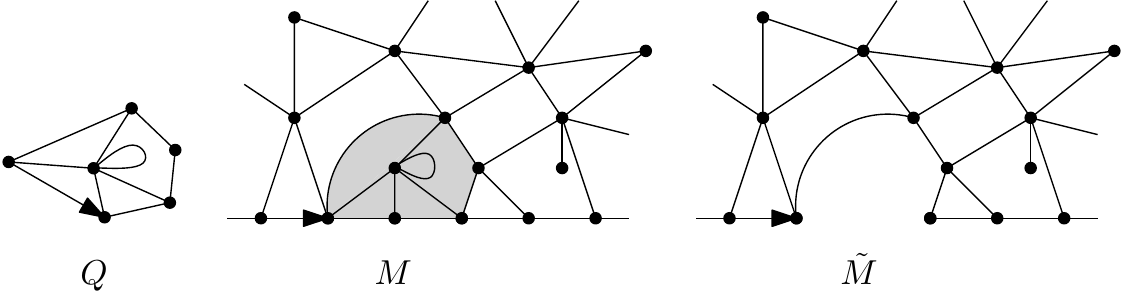} }
\caption{An illustration of domain Markov Property. Left: A finite simply connected map $Q$. 	Centre: A part of $M$ containing $Q$ with $2$ edges along the boundary. Right: The resulting map $\tilde{M}$ after removal of $Q$. Domain Markov property states that the law of $\tilde{M}$ is the same as that of $M$.}\label{fig: dmp}
\end{figure}

In \cite{AR13}, measures on half planar maps were considered which satisfy two natural properties: \textbf{translation invariance} and \textbf{domain Markov property}. See \cite[Section 1.1]{AR13} for the precise definitions. Roughly, the first property states that the law of the map is invariant to translation of the root edge along the boundary and the second states that if we condition that the map contains a fixed simply connected finite map which contains the root edge, then the law of the remaining map is the same. The latter is illustrated in Figure \ref{fig: dmp}. The following theorem of \cite{AR13} characterizes all such probability measures.

\begin{thm}[\cite{AR13}]\label{thm:main}
  All translation invariant and domain Markov measures supported on half planar triangulations without self-loops form a one parameter family $\H_{\alpha}$ where the parameter $\alpha \in [0,1)$. Furthermore $\alpha$ denotes the probability of the event that the triangle adjacent to the root edge is incident to an internal vertex.
\end{thm}

\subsection{About the proofs}
Let us now give some brief intuition behind \cref{T:speed,T:return_prob}.
 It can be shown (see \cite{LP:book}, Proposition 6.9) that graphs with positive Cheeger constant and at most exponential edge volume growth (meaning number of edges in the combinatorial ball of radius $r$ is at most $C^r$ for some $C>1$) has positive liminf speed. It can also be shown that on graphs with positive cheeger constant, the return probability is roughly $e^{-cn}$ for some $c>0$ (see \eqref{eq:exp_return}).

Our maps $H$ however are random and any given fixed configuration occurs somewhere in the map with probability $1$. Hence we need to consider a parameter more robust to such small perturbations called anchored expansion introduced by Benjamini, Lyons and Schramm \cite{BLS99}.\ Roughly, a graph has positive anchored expansion if, instead of all finite sets, all connected sets containing a fixed vertex in the graph have large boundary compared to its volume (details in \cref{sec:anchored}). It was shown by Vir\'{a}g in \cite{B00} that graphs with anchored expansion and bounded degree has positive speed away from the root and the return probability is at most $\exp(-cn^{1/3})$ for some constant $c>0$ (for details see \cref{sec:anchored}). Our maps $H$ do have anchored expansion almost surely as was illustrated in \cite{Ray13}, Theorem 2.3. However, since the maps $H$ do not have uniform bounded degree, even positive speed away from the starting point in $H$ do not follow directly from Vir\'{a}g's result in \cite{B00}. It is not difficult to obtain graphs with anchored expansion, exponential edge volume growth, unbounded degree and zero speed. However Theorem 3.5 of \cite{T92} does ensure that $H$ is transient almost surely just because $H$ has anchored expansion. Similar results for return probabilities on Cayley graphs were obtained in \cite{V91,HSC93}.

Our proof methodology however relies on the techniques used by Vir\'{a}g in \cite{B00}. In \cite{B00}, it was shown in Proposition 3.3 that every graph with anchored expansion contains a subgraph with positive Cheeger constant. The complement of this subgraph are sets with small boundary which Vir\'{a}g called islands. The trick is to control the distribution of these islands and make sure the walker does not spend too much time on the islands. We aim to follow a similar approach here and use the geometry of $H$ to arrive at the proposed results.

Recently a full-plane version of $H$ was constructed by Curien in \cite{PSHIT}. Our second main tool is a coupling which realizes $H$ as a submap of its full plane version. Hence we can use the properties of random walk on the full-plane version to our advantage. For example, it was shown in \cite{PSHIT} that random walk such full plane hyperbolic triangulations have positive speed almost surely.

Let us finish by mentioning that there has been growing interest in studying simple random walk on
random planar maps in recent years (see
\cite{GN12,bjornberg2013recurrence,BC11,BeSc}). For example, it was an
open question for about a decade whether uniform infinite planar
triangulations given by \cite{UIPT1} was recurrent or transient and
only very recently it was resolved in \cite{GN12}. However, many
questions do remain open about the behaviour of simple random walk on
the uniform infinite planar maps. For example, it is not known what is
the speed of the simple random walk on these maps, and only an upper
bound is provided in \cite{BC11}.

\paragraph{Disclaimer:} In the computations that follow, the constants might change from one line to next but we shall still denote them by the same letter $c$ for clarity. Also, we fix an $\alpha \in (2/3,1)$ throughout the rest of the paper and we shall drop the subscript $\alpha$ from the notation $\H_\alpha$ unless otherwise stated. Throughout, $H$ will denote a half plane map with law $\H$.

\paragraph{Organization:}
In \cref{sec:expansion,sec:anchored} we recall some of the background
results about expansion, anchored expansion and the paper of Vir\'{a}g
\cite{B00}.  In \cref{sec:dmp,sec:peeling,sec:geom}, we recall some of the
results about the geometry of domain Markov triangulations, and derive some
of their consequences.  In \cref{sec:PSHIT}, we recall the stochastic
hyperbolic full plane triangulations of Curien \cite{PSHIT} and prove the
coupling of $\H$ as its sub-map.  The knowledgeable reader could skim these
and proceed to \cref{sec:3}, where we prove \cref{T:speed} and \cref{sec:4}
where we prove \cref{T:return_prob}.  We end with some comments and open
questions in \cref{sec:open}.

\section{Background}\label{sec:back}

We begin by reviewing various definitions we use. The reader familiar with
anchored expansion and the work of Vir\'ag \cite{B00} may skip to
\cref{sec:dmp}. The reader familiar with the domain Markov property and
\cite{AR13} may skip \cref{sec:dmp,sec:peeling}.

\subsection{Graph notations and expansion}
\label{sec:expansion}

In this section we review the notion of expansion and some properties of
graphs with positive Cheeger constant. Let us start with a few definitions.
A \textbf{weighted graph} $G$ is a graph along with a positive weight
$w(u,v)$ assigned to every edge $(u,v)$. The weight of a vertex $u$ is the
sum of the weights of the edges incident to it and is denoted $w(u)$. We
can view an unweighted graph also as a weighted graph with every edge
having weight $1$. The random walk on a weighted graph is a Markov chain on
its vertices with transition probabilities given by
\[
p(x,y) =\frac{ w(x,y)\mathbbm{1}_{x\sim y}}{w(x)}.
\]

A sequence $(X_n)$ in a graph $G$ is said to have {\bf positive liminf
  speed} if
\[
\liminf_{n\to\infty} \frac{d^G(X_n,\rho)}{n} > 0,
\]
where $d^G$ denotes the graph distance, ignoring weights, and $\rho$ is
some fixed vertex (this condition is clearly independent of $\rho$).

We consider the Hilbert space of functions defined on the vertices of
$G$ with the inner product and norm
\[
\langle f,g \rangle = \sum_{v \in V(G)}f(v)g(v)w(v), \quad \text{ and }
\quad \|f\|= \langle f,f \rangle^{1/2}.
\]
The Markov transition kernel $P$ for the random walk $\{X_n\}_{n\ge 1}$ on
$G$ is the operator defined by $Pf(v) = \E(f(X_1)|X_0=v)$.
\
For a set $S \subset V(G)$, the {\bf edge boundary} of $S$, denoted
$\partial S$ is defined to be the set of edges which have one endpoint in
$S$ and the other in $V(G)\setminus S$. The {\bf Cheeger constant} of $G$
is defined to be
\[
i = i(G) := \inf_{S \subset V(G)} \frac{|\partial S|}{|S|}
\]
where the infimum is over all finite subsets of vertices of $G$ and for any
set $X$ of vertices or edges, $|X|$ denotes the sum of weights over $X$.
Notice that the weight of a set of vertices $X$ in an unweighted graph is
nothing but the sum of the degrees of the vertices in $X$.

Recall Cheeger's classical inequality (see \cite{Dodziuk84,Mohar88,LP:book}
or \cite[Proposition 4.1]{B00}):
\begin{equation}
  \|P\| \le1-\frac {i^2}{2}. \label{prop:DM}
\end{equation}
Inequality \eqref{prop:DM} implies
\begin{equation}
  \P(X_n =\rho) = \frac{1}{w(\rho)} \langle\mathbbm{1}_\rho, P^n
  \mathbbm{1}_\rho \rangle
  \le \frac{1}{w(\rho)}\|\mathbbm{1}_\rho\|^2.
  \|P\|^n\le(1-i^2/2)^n\label{eq:exp_return}
\end{equation}
Thus on a graph with positive Cheeger constant, the return probabilities decay
exponentially in the number of random walk steps.

\subsection{Anchored expansion}\label{sec:anchored}

In the context of random graphs, frequently it is the case that expansion
holds on average but not uniformly. In many such cases we have a weaker
form knows as {\bf anchored expansion}. We follow the terminology and
reasoning of Vir\'ag \cite{B00}, and repeat some definitions and results
from there. We say a weighted graph $G$ has anchored expansion if
$i^*(G)>0$, where the anchored expansion constant is defined by
\[
i^*(G) := \liminf_{n \to \infty} \left\{ \frac{|\partial S|}{|S|} \ :\
  \text{$S\subset V(G)$ is connected, $|S|=n$ and $\rho\in S$}\right\}
\]

As mentioned in the introduction, it was proved by Vir\'{a}g (\cite{B00})
that for bounded degree unweighted graphs, anchored expansion implies
positive liminf speed. This result is not applicable in our setting because
there is no uniform upper bounds on the degrees of the vertices of the
graphs we consider. However, some information about the geometry of graphs
with anchored expansion as established by Vir\'{a}g will be useful in our
subsequent analysis. A key idea is to prove that a graph with anchored
expansion contains a subgraph (called the ocean) with positive Cheeger
constant. The complement of this subgraph has only finite components which
Vir\'{a}g calls islands. The main idea of Vir\'{a}g was to show that the
random walk cannot spend too much time in the islands.

For a finite set $S$ of vertices and a number $i \in (0,1)$, define the
\textbf{$i$-isolation} of a finite set $S$ to be
\[
\Delta_iS = i|S| - |\partial S|
\]
A set of vertices with positive $i$-isolation is called $i$-isolated.  A
finite vertex set $S$ is called an \textbf{$i$-isolated core} if $\Delta_i
S > \Delta_i A$ for every $A \subsetneq S$. Putting $A$ to be the empty
set, we see that an $i$-isolated core is always $i$-isolated.

\begin{prop}[{\cite[Corollary 3.2]{B00}}]
  \label{prop:isolated core}
  The union of finitely many $i$-isolated cores is an $i$-isolated core.
\end{prop}

Suppose $0<i<i^*(G)$ and let $A_i$ be the union of all $i$-isolated cores
in $G$. Anchored expansion and \cref{prop:isolated core} imply that
$i$-isolated cores containing any vertex have a bounded size, and so all
connected components of $A_i$ are finite unions of $i$-isolated cores and
hence are finite. The connected components of $A_i$ are called {\bf
  $i$-islands}. Components of the complement $G \setminus A_i$ are
called \textbf{$i$-oceans} (note $G \setminus A_i$ is not necessarily
connected, since cutsets need not be connected). The following proposition
shows that if $i'<i$ and $G$ is an unweighted graph, all small $i$-islands
``sink'' in the $i'$-ocean.

\begin{prop}[{\cite[Lemma~3.4]{B00}}]
  \label{prop:sinking}
  Suppose $G$ is an unweighted graph and $0<i'<i<i^*(G)$. Then $A_{i'}
  \subset A_i$. Further if $B$ is an $i$-island of size at most $1/i'$.
  Then $B \subset G \setminus A_{i'}$
\end{prop}

We now describe an {\bf induced random walk}.  Let $\{X_n\}_{n \ge 1}$ be a
random walk on a graph $G$ and suppose $0<i<i^*(G)$. Let
$(\tau_j)_{j\geq0}$ be the ordered set of times at which $X_t$ is in an
$i$-ocean. Then $(X_{\tau_j})$ is a reversible Markov chain on $G\setminus
A_i$, and in particular is equivalent to a random walk on a graph $G_i$,
which is closely related to the oceans, defined as follows. The vertices of
$G_i$ are $G \setminus A_i$. For any vertex $v$, its weight $w_i(v) = w(v)$
is inherited from $G$. For any pair of vertices $u,v$, put a weight
\[
w_i(u,v) = w_i(u) \P(X_{\tau_1} = v | X_0=u)
\]
Note that $w_i$ is a symmetric function on the edges because of the
reversibility of the walk. The following is proved in \cite{B00}, following
Lemma~3.4.

\begin{prop}[\cite{B00}]
  \label{prop:cheeger_weighted}
  The graph $G^i$ has cheeger constant at least $i$.
\end{prop}

We say a graph $G$ has \textbf{upper exponential growth} if $|B_r| \leq
e^{Cr}$ for some $C<\infty$, where $B_r=\{x:d^{G}(x,\rho) \le r\}$ is the
ball. As usual, this does not depend on the choice of $\rho$. This is
connected to the random walk behaviour through the following:

\begin{prop}\label{lem:island_walk}
  Let $G$ be a graph with anchored expansion and upper exponential growth.
  Fix $i<i^*(G)$ and let $(\tau_m)$ be as above. Then
  \[
  \liminf_{m\to\infty} \frac{ d^G(X_{\tau_m},\rho) }{m} >0
  \]
\end{prop}

\begin{proof}
  This is a straightforward application of Lemma~4.2 of \cite{B00} to
  $G_i$, with $f(x) = d^G(x,\rho)$ (as opposed to the distance in $G_i$).
\end{proof}

\subsection{Domain Markov triangulations}
\label{sec:dmp}

We now recall some definitions and properties of domain Markov half planar
triangulations, which are the main object in this paper. We refer the
reader to \cite{AR13} for further details.

As mentioned in \cref{thm:main}, $\alpha$ is the probability under
$\H_\alpha$ of the event that the triangle incident to the any given
boundary edge contains the edge and one internal vertex.
The other possibilities are referred to as steps of the form $(L,i)$ or
$(R,i)$. A step of form $(L,i)$ (resp.\ $(R_i)$), is the event that the
triangle incident to some fixed boundary edge has its third vertex on the
boundary at a distance $i$ to the left (resp.\ right) of the given edge
along the boundary (see \cref{fig:forms}). We shall also talk about such
events with the root edge replaced any fixed edge on the boundary of the
map.

\begin{figure}
  \centering{\includegraphics[width=.8\textwidth]{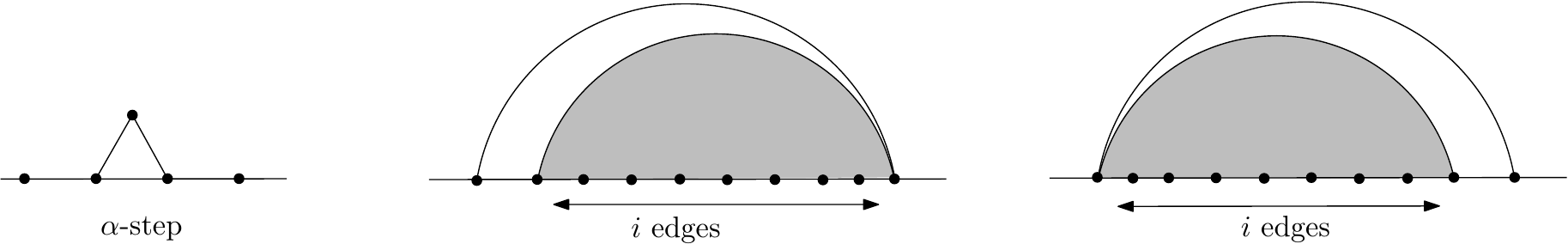}}
  \caption{Left: An $\alpha$-step. Centre: A step of the form
    $(R,i)$. Right: A step of the form $(L,i)$. The gray area denotes some
    unspecified triangulation.}
  \label{fig:forms}
\end{figure}

Because of translation invariance, the measures of such events do not
depend on the choice of the boundary edge. It was also shown in \cite{AR13}
that for any fixed $i \ge 1$, the probability of $(L,i)$ and $(R,i)$ are
the same (denoted $p_i$). Let $p_{i,k}$ denote the probability of the event
that $(L,i)$ (or $(R,i)$) occurs and the triangle incident to the root edge
separates $k$ internal vertices from infinity. Let $\beta=p_{1,0}$ be the
probability of the event of the form $(R,1)$ with no internal vertex in the
$2$-gon enclosed by the triangle incident to the root edge. The following
formulas were derived in \cite{AR13}:
\begin{align}
  \beta & = \frac{\alpha(1-\alpha)}{2}&
  p_{i,k} & = \phi_{k,i+1} \beta^i(\alpha \beta)^k&
  p_i & \sim c \left(\frac{2}{\alpha}-2\right)^{i} i^{-3/2}\label{eq:super_pi}
\end{align}
for some constant $c>0$, where $\phi_{k,i+1}$ denotes the number of
triangulations of an $(i+1)$-gon with $k$ internal vertices.
Notice that $p_i$ has an exponential tail for any $\alpha>2/3$.

\subsection{Peeling}\label{sec:peeling}

Let us briefly describe the concept of peeling, originating in work of
Watabiki \cite{Wat} and which was introduced in its present form in
\cite{UIPT2}. For further background, history and numerous applications of
this very useful tool, we refer to
\cite{UIPT2,AC13,curien13glimpse,menard2013perc,AR13,BC11}.

Given a map $T$, the idea is to construct a growing sequence of simply
connected sub-maps $P_n$ with complements $T_n = T\setminus P_n$ as
follows: at every step, we pick an edge in the boundary of $T_n$. We then
construct $P_{n+1}$ by adding to $P_n$ the face $f$ of $T_n$ incident to
the chosen edge along with any finite components of $T\setminus (P_n \cup
f)$.  This can be carried out for full plane maps, half plane maps, and
other topologies as well.

If $T$ is a domain Markov half plane triangulation with law $\H$ and the
edge chosen at each step is independent of $T_n$, then $T_n$ has the same
law $\H$ for all $n$, and is independent of the revealed map $P_n$. Notice
that every peeling step is of the form $(L,i)$ or $(R,i)$ as discussed in
\cref{sec:dmp} and if $T$ is a full-plane map and $P_n$ is a triangulation
of a $p$-gon, then necessarily $i\le p-2$ since we do not allow self-loops
in our triangulation.

The nice thing about peeling is that we are free to choose the edge on which we perform the next peeling step, the only constraint being the choice should depend only upon $P_n$ (along with possibly another source of randomness). We discuss two such algorithms which we shall need.

\paragraph{Peeling to reveal hulls:} The aim of this algorithm is to sequentially reveal the graph distance balls of radius $r$ around the root (or any finite subset of the boundary) along with the finite components of the complement. The hull of the ball of radius $r$ is completed at some random time $N_r$. Having obtained $P_{N_r}$, we peel continuously on the edges incident to vertices on the boundary of $P_{N_r}$ which is common with $T_{N_r}$ until every vertex on the boundary of $P_{N_r}$ is no longer in the boundary. We refer to \cite{Ray13}, Section 4.1 for more details.

\paragraph{Peeling along random walk:} This was introduced in \cite{BC11}. The idea is to run a random walk simultaneously with the peeling procedure. Having known $P_n$ and the vertex $v$ the random walker is in, there are two possible steps. If $v$ not in the boundary of $P_n$, we just perform a random walk step. Otherwise we continue peeling until we reveal the hull of the ball of radius $1$ around $v$ in $T_n$. Then we perform a random walk step. For details we refer to \cite{BC11}, Section 1.4.

\subsubsection{Free triangulations}

If we perform peeling steps in $H$, then it turns out via the calculations done in \cite{AR13} that the finite triangulation enclosed by the revealed triangle after a step of the form $(L,i)$ or $(R,i)$ is distributed as a \textbf{free triangulation} (defined below) of an $(i+1)$-gon with parameter $\alpha \beta$ with $\beta$ given by \eqref{eq:super_pi}. For details we refer to \cite{UIPT2,Ray13}.

\begin{defn}\label{def:free}
  The \textbf{free} distribution on rooted triangulations
  of an $m$-gon with parameter $q\leq\frac{2}{27}$ is the probability
  measure that assigns weight $q^n / Z_m(q)$ to each rooted
  triangulation of the $m$-gon having $n$ internal vertices, where
  \[
  Z_m(q) = \sum_n \phi_{n,m} q^n .
  \]
and $\phi_{n,m}$ is the number of triangulations of an $m$-gon with $n$ internal vertices.
\end{defn}

The partition function in \cref{def:free} can be explicitly computed:
\begin{prop}[\cite{GFbook}]\label{prop:Z}
  Let $q = \theta(1-2\theta)^2$. Then for $m
\ge 0$, $Z_{m+2}(q)$ is finite if and only if $q \in [0, 2/27]$ (equivalently $\theta \in [0,1/6]$) and is given by
  \[
  Z_{m+2}(q) = ((1-6\theta) m +2-6\theta) \frac{(2m)!}{m!(m+2)!}
  (1-2\theta)^{-(2m+2)}.
  \]
\end{prop}

Let $Y=0$ with probability $\alpha$ and $Y=i$ with probability $p_i$ (as given in \eqref{eq:super_pi}) for $i \ge 2$. Let $I_m$ be the number of internal vertices of a free triangulation of an $m$-gon with parameter $\alpha \beta$ for $m\ge 2$ and $I_1=1$. We say a variable $X$ has \textbf{exponential tail} if there exists a constant $c>0$ such that $\P(|X|>i) <e^{-ci}$. We next show that $Y+I_{Y+1}$ has exponential tail.

\begin{lemma}\label{lem:exp_tail}
Let $Y,I_{Y+1}$ be as above. There exists a $\lambda>0$ such that
\[
 \E\left(e^{\lambda(Y+I_{Y+1})}\right) <\infty
\]
 In particular, the number of edges added in each peeling step has exponential tail.
\end{lemma}
\begin{proof}
 In light of \eqref{eq:super_pi}, we need to show
\begin{equation}
 \sum_{i\ge 1} e^{\lambda i}\beta^i\sum_{k \ge 0}e^{\lambda k}\phi_{i+1,k}(\alpha\beta)^k <\infty
\end{equation}
But $\alpha\beta <2/27$ if $\alpha \in (2/3,1)$ via \eqref{eq:super_pi}. Hence the sum over $k$ is finite if $\lambda>0$ is small enough via \cref{prop:Z}. Also since $\beta<1/9$,
\[
 \sum_{i\ge 1} e^{\lambda i}\beta^iZ_{i+1}(e^\lambda\alpha \beta)
\]
is also finite for small enough choice of $\lambda$ via the expression of the partition function given by \cref{prop:Z}. The last sentence in the lemma follows from the fact that the number of edges added in each peeling step is dominated by $2Y+3I_{Y+1}-1$ via Euler's formula.
\end{proof}

\subsection{Geometry of supercritical triangulations}\label{sec:geom}

Here we restate some results from \cite{Ray13,AR13} where the supercritical
half planar triangulations were introduced and studied. We start with a
statement about the probability of finite events in $\H_\alpha$, which
could also be used as an alternative definition of $\H_\alpha$.

\begin{lemma}[\cite{AR13}]\label{lem:prob}
  Let $Q$ be a simply connected triangulation with a simple boundary, with
  some marked connected segment of $\partial Q$.  Let $H$ be a half plane
  triangulation, and consider the event $A_Q$ that $Q$ is a sub-map of $H$
  with the marked segment being the only intersection of $Q$ with $\partial
  H$. Then
  \[
  \H(A_Q) =  \alpha^{V(Q)}\beta^{F(Q)-V(Q)}
  \]
  where $V(Q)$ is the number of vertices of $Q$ not on $\partial H$ and
  $F(Q)$ is the number of faces of $Q$.
\end{lemma}

Recall the definition of $i$-isolated sets from \cref{sec:anchored}: a set
$S$ is $i$-isolated if $i|S|-|\partial S|$ is positive.

\begin{lemma}[{\cite[Proposition 4.10]{Ray13}}]\label{lem:i-isolated}
  For any $\alpha>2/3$ there exists a constant $i=i(\alpha)\in (0,1)$ such
  that the probability that there exists an $i$-isolated connected set of
  vertices $S$ containing the root vertex in $H$ with $|S|\geq n$ is at most
  $\exp(-cn)$.
\end{lemma}

\cref{lem:i-isolated} along with Borel-Cantelli gives the following corollary:

\begin{corollary}[\cite{Ray13}]\label{cor:Hanch}
  Let $i(\alpha)$ be as in \cref{lem:i-isolated}. Then $\H_\alpha$-almost
  surely, the map $H$ has anchored expansion constant $i^*(H) \ge
  i(\alpha)>0$.
\end{corollary}

The following lemma controls the probability of the ball volumes in $H$
being atypically large or small.

\begin{lemma}\label{L:vol}
  Let $B_r$ denote the hull of the graph distance ball of radius $r$ around
  the root vertex in $H$. There exists constants $a,b>1$ and $c>0$
  depending only on $\alpha$ such that
  \[
  \H_\alpha\big( a^r \leq |B_r| \leq b^r \big) \geq 1 - e^{-c r}.
  \]
\end{lemma}

This is closely related to a statement from \cite{Ray13} concerning the
a.s.\ asymptotic behaviour of $|B_r(H)|$.  The arguments below are
variations on arguments used there.

\begin{proof}
  First we prove the lower bound.
  Choose $i=i(\alpha)$ as in \cref{lem:i-isolated}. Clearly, $|B_{r/2}|\ge
  r$ (counting just the boundary vertices).  By \cref{lem:i-isolated}, the
  probability that there is some $i$-isolated set of size at least $r$ containing the root vertex is
  at most $e^{-cr}$ for some $c>0$. Assume there is no such set, then for
  any $k\geq r/2$ we have $|\partial B_k| \geq i|B_k|$. However, each edge
  in $\partial B_k$ contributes $1$ to the weight of some vertex in
  $B_{k+1}\setminus B_k$.  Thus on the event that there is no large
  isolated set we have $|B_{k+1}| \geq (1+i)|B_k|$, which implies $|B_r|
  \geq r(1+i)^{r/2}$. This yields the lower bound with $a=\sqrt{1+i}$.

  \medskip

  For the upper bound, notice that by Markov's inequality it suffices to
  prove $\E(|B_r|) < C^r$ for some $C$.  For this, we use that $|B_r| \leq
  2\#E(B_{r+1})$, where $E(B_{r+1})$ is the set of edges in $B_{r+1}$, and
  bound the expectation of the number of edges in the ball.

  Recall the stopping time $N_r$ from \cref{sec:peeling} which denotes the
  time taken to reveal $B_r$ during the peeling process to reveal hulls. It
  is clear from the description of the peeling process and
  \cref{lem:exp_tail} that $\#E(B_r)$ is a sum of $N_r$ many i.i.d.\
  variables with finite expectation (the number of edges added on each
  step). Since $N_r$ is a stopping time, Wald's identity yields $\E \#
  E(B_r) < c \E N_r$. Since a geometric number of peeling steps is required
  to swallow each vertex on the boundary of $B_r$, we can use the crude
  estimate that $N_{r+1}-N_r$ is dominated by $\# E(B_r)$ many geometric
  variables. Thus $\E(N_{r+1}-N_r) < c'\E \# E(B_r)$. Putting together the
  pieces,
  \[
  \E(N_{r+1}-N_r) < c'\E\# E(B_r) < c c' \E N_r.
  \]
  This implies $\E(N_r) < C^r$ for some $C$, and the desired bound.
\end{proof}

\begin{lemma}\label{lem:1ball}
  Let $B_r$ be the hull of the ball of radius $r$ around the root.  There
  exists a constant $c>0$ such that for all $n\ge 1$
  \[
  \H(|B_1| > n) \leq e^{-cn}.
  \]
\end{lemma}

\begin{proof}
  We use the peeling process to reveal hulls as described in
  \cref{sec:peeling}.  Notice that it is enough to show the exponential
  tail for the number of edges in $B_2$, since $|B_1|$ is at most twice
  that number.

  \cref{lem:exp_tail} implies that number of edges added in each peeling
  step has an exponential tail.  This along with the fact that it takes
  geometric number of steps to reveal $B_1$ imply that $\#E (B_1)$ has
  exponential tail. The domain Markov property tells us that if we continue
  peeling to reveal the neighbourhood of any given vertex on the boundary
  of $B_1$, the number of edges added also has exponential tail via the
  previous argument. Thus the total number of edges added after revealing
  $B_1$ is at most a sum of $\#E(B_1)$ many variables with exponential tail,
  and hence $\#E(B_2)-\#E(B_1)$ also has exponential tail.
\end{proof}

Another result we quote
from \cite{Ray13} is that the graph distance between two vertices on the boundary is
at least linear in their corresponding distance along the boundary
with high probability. Let us enumerate the boundary vertices as $\{v_i\}_{i \in Z}$ with $v_0$ denoting the root vertex and $v_{i}$, $v_{-i}$ denoting the vertices at a distance $|i|$ along the boundary from the root vertex.
\begin{prop}[\cite{Ray13}, Lemma 4.6]\label{prop:dist_along_bdry}
There exists a constant $t=t(\alpha)>0$ such that
\[
\H(d^H(v_i,v_j) <t |i-j|) < e^{-c|i-j|}
\]
for some $c>0$ depending only on $\alpha$.
\end{prop}

\begin{lemma}\label{lem:hull1}
  Consider a fixed connected segment $S$ with $n$ vertices on the boundary
  of $H$. Then for every constant $b>0$ there exists a constant $\gamma>0$ such
  that for all $n \ge 1$,
  \[
  \H\left(|S| > \gamma n \right) < \exp(-bn)
  \]
\end{lemma}

\begin{proof}
  We prove a stronger result: the number of edges in the hull of the
  neighbourhood of radius $1$ around $S$ has such an exponential tail. We
  use the peeling procedure to reveal hulls as in \cref{sec:peeling} and
  borrow the notations from there. In each step, we choose a vertex $v \in
  S$ which is in $T_n$ and peel until we reveal the hull of the ball of
  radius $1$ around $v$ in $T_n$.  The domain Markov property ensures that
  the number of edges added in such steps are i.i.d.  \cref{lem:1ball}
  implies these have an exponential tail. Since there are at most $n$ such
  steps, the lemma follows by a standard large deviation estimate.
\end{proof}

As a corollary of \cref{prop:dist_along_bdry,lem:hull1}, we get
\begin{corollary} \label{cor:ball_boundary}
There exists a $t' = t'(\alpha)$ such that
 \[
    \H(|B_n(H)\cap \partial H| >t'n)<\exp(-cn)
 \]
for some constant $c>0$ depending only on $\alpha$.
\end{corollary}

\subsection{Stochastic hyperbolic triangulations}
\label{sec:PSHIT}

In \cite{PSHIT}, Curien constructs a one parameter family of measures which
we denote $\mathbb F_\kappa$ for $\kappa \in (0,2/27]$ which are supported
on full plane triangulations.  These are full plane analogues of the half
plane maps $\H_\alpha$ considered in this paper, where $\kappa=\alpha\beta
= \alpha^2(1-\alpha)/2$.  There is a close connection between the
half-plane and full-plane hyperbolic triangulations which allows us to make
use of some of Curien's results.  We denote a sample of $\mathbb F_\kappa$ by
$F_\kappa$, or just $F$.

We shall need some properties of $\F_\kappa$, stated in \cite{PSHIT}.  By
its definition, for a finite simply connected triangulation $t$ with simple
perimeter $p$ we have that $\F_\kappa(t\subset F) = C_p \kappa^{|t|}$,
where $|t|$ is the number of vertices in $t$ (for notational clarity, throughout this section $|.|$ denotes the number of vertices), and $C_p$ is some sequence of
positive numbers (depending implicitly on $\kappa$).  Moreover, if $\tilde
C_p = \beta^p C_p$ with $\beta=\frac{\alpha(1-\alpha)}{2} =
\frac{\kappa}{\alpha}$ then $\tilde C_p$ is increasing and converges.

The following is noted without proof in \cite{PSHIT}, and we
include a proof here.

\begin{lemma}\label{lem:coupling}
  Set $\kappa = \frac{\alpha^2(1-\alpha)}{2}$. There exists a coupling
  between $\H_\alpha$ and $\mathbb F_\kappa$ such that almost surely,
  $H_\alpha \subset F_\kappa$, where the inclusion need not map the root of
  $H$ to the root of $F$.
\end{lemma}

\begin{lemma}\label{lem:comp_conv}
  Let $\{t_i\}_{i\ge 1}$ be a sequence of finite triangulations where $t_i$
  is a triangulation of a $p_i$-gon and $p_i \to \infty$ as $i \to
  \infty$. Let $\mathcal N_i$ denote the event that $t_i$ can be realized
  as a submap of $F$ with coinciding roots. Conditioned on $\mathcal N_i$,
  the distribution of $F \setminus t_i$ with a given root on the boundary
  of $t_i$ converges in distribution to $H$ as $i\to \infty$.
\end{lemma}

\begin{proof}
  Suppose $t\subset t'$ are two finite triangulations with simple
  boundaries of length $p,p'$, and that $t'$ is constructed by gluing to
  $t$ some finite triangulation $q$ along some segment of $t$'s boundary. A
  consequence of the formula for $\F_\kappa(t\subset F)$ is that
  \[
  \F_\kappa(t'\subset F | t\subset F)
  = \frac{C_{p'}}{C_p} \kappa^{|t'|-|t|}
  = \frac{C_{p'}}{C_p} \kappa^{|q|}.
  \]
  In particular, we see that the probability of containing $t\cup q$
  conditioned on containing $t$ depends on $t$ only through $p$ (i.e.\ the
  law of $F\setminus t$ depends only on $p$. Moreover, as $p\to\infty$ we
  get that the conditional probability of containing $q$ tends to
  $\beta^{p'-p}\kappa^{|q|}$. From the relation $\kappa=\alpha\beta$ we see
  that this is exactly the $\H_\alpha(q\subset H)$.
\end{proof}

\begin{lemma}\label{L:boundary_growth}
  If $(T_i)$ is a sequence of finite sub-triangulations of $F$, each
  generated by peeling at one edge of the previous, then for some $c>0$, a.s.\
  $|\partial T_i| > c i$ for all large enough $i$.
\end{lemma}

\begin{proof}
  Let $X_p$ be the increment in the boundary size when performing one
  peeling step on a triangulation with boundary size $p$ (so that
  $X_p\in\{1,-1,-2,-3,\dots\}$). Let $X_\infty$ take the value $1$ with
  probability $\alpha$ the value $-i$ with probability $p_i$ for $i \ge 1$.
  From the above discussion we have   that 
  \[
  \P(X_p = -i)
  = \frac{C_{p-i}}{\beta^i C_p} \P(X_\infty=-i)
  = \frac{\tilde C_{p-i}}{\tilde C_p} \P(X_\infty=-i).
  \]
  Since $\tilde C_p$ is increasing, we deduce that $X_p$ stochastically
  dominates $X_\infty$.

  It is known (\cite[Lemma~4.2]{Ray13}) that $\E X_\infty>0$. We therefore have that
  $|\partial T_i|$ is a Markov chain with steps that stochastically
  dominate i.i.d.\ copies of $X_\infty$. The claim follows by the law of
  large numbers.
\end{proof}

\begin{proof}[Proof of \cref{lem:coupling}]
  The idea is to show that it is possible to perform infinitely many
  peeling steps in $F_\kappa$ so that some infinite component of the map
  remains unexplored, and that the unexplored region has law
  $\H_\alpha$. By the \cref{lem:comp_conv}, this will hold as long as the
  boundary of the revealed maps tends to infinity, and some of $F$ remains
  unrevealed.

  We will perform the peeling procedure to reveal a growing sequence of
  neighbourhoods $\{P_n\}_{n\ge 1}$ around the root of $F$ as follows. Having defined
  some edge $e_n$ in the boundary of $P_n$, we peel at the edge {\em farthest}
  from $e_n$ along the boundary to get $P_{n+1}$. If $e_n$ is also on the
  boundary of $P_{n+1}$ then set $e_{n+1}=e_n$. Otherwise, pick
  arbitrarily some new edge on the boundary to be $e_{n+1}$.

  Let $A_n$ be the event that $|\partial P_n| \geq cn$, with $c$ from
  \cref{L:boundary_growth}. On $A_n$, the probability that $e_n$ is
  swallowed in the peeling step is exponentially small.  By Borel-Cantelli,
  a.s., there are only finitely many $n$ for which $A_n$ holds and $e_n$ is
  swallowed.  By \cref{L:boundary_growth}, $A_n$ a.s.\ holds for all but
  finitely many $n$.  Thus $e_n$ is eventually constant, and we are done.
\end{proof}

\section{Positive speed away from the boundary}\label{sec:3}

In this section we prove \cref{T:speed}. Random walks in $H$ will always
start from the root vertex $\rho$ unless otherwise stated.  Our strategy is
to prove that the simple random walk in $H$ hits the boundary finitely
often almost surely and then use the coupling in \cref{lem:coupling} to get
positive liminf speed away from the root.  Using this we then prove
positive liminf speed away from the boundary by an application of the
Carne-Varopoulos bound.  To carry out this plan we first prove in
\cref{lem:weak_speed} a weaker lower bound of $n/\log^3n$ on
$d^{H}(X_n,\rho)$.

\medskip

Curien proves in \cite[Theorem 3]{PSHIT} that the simple random walk on
$F_\kappa$ has positive speed almost surely for any $\kappa \in (0,2/27)$.
His proof is rather indirect, though we note that it is possible to prove
this more directly using the ideas of \cite{B00}, with many of the
ingredients already appearing in \cite{PSHIT}: anchored expansion, and
upper exponential growth give positive speed for the time in the ocean;
ergodicity of $F$ w.r.t.\ the random walk implies that a positive fraction
of time is in the ocean, hence positive speed for the random walk.

\begin{lemma}\label{lem:pos_prob}
  Let $(X_n)$ be a simple random walk on $H$ starting from the root vertex
  $X_0=\rho$.  There exists a constant $s=s(\alpha)>0$ such that almost
  surely either $\liminf{d^H(X_n,\rho)/n \geq s}$ or else there exits an
  $n\ge 1$ such that $X_n \in \partial H$.
\end{lemma}

\begin{proof}
  Using the coupling in \cref{lem:coupling}, we can couple $H$ as a submap
  of $F$ almost surely.  If $X_1 \in \partial H$ we are done, so assume
  that we are on the event $X_1 \in H \setminus \partial H$.  We couple the
  random walk $X$ on $H$ with a random walk $Y$ on $F$, both starting from
  $X_1$.

  Let $\tau = \inf\{t \ge 1: X_t \in \partial H\}$.
  Since $H \subset F$ and $X_1 \in H \setminus \partial H$, we can couple
  $\{X_n\}$ and $\{Y_n\}$ such that $X_i = Y_i$ for $1\le i \le \tau$. If
  $\tau<\infty$ almost surely, we are done. If $\tau=\infty$, Theorem 3 of
  \cite{PSHIT} ensures that $\lim d^F(Y_n,Y_0)/n =s >0$ for some $s =
  s(\alpha)$ (note that the speed does not depend on the starting vertex of
  the random walk, nor on the vertex from which the distance is measured.)
  Since distances in $H$ are greater than distances in $F$ between the same
  vertices, we conclude $\liminf d^H(X_n,X_0) \geq s$ almost surely on the event
  $\{\tau = \infty\}$.
\end{proof}

Since $H$ almost surely possesses anchored expansion (\cref{cor:Hanch}), we can
decompose $H$ into islands and oceans as described in \cref{sec:anchored}.
For a graph $G$ with anchored expansion constant larger than $i$, recall
the weighted graph $G_i$ constructed in \cref{sec:anchored}. Using
\cref{cor:Hanch,prop:cheeger_weighted}, we conclude:

\begin{prop}\label{prop:cheeger}
  Let $H$ have law $\H_\alpha$, and $i(\alpha)$ be as in
  \cref{lem:i-isolated}. For any $i\le i(\alpha)$, the graph $H_i$ has
  cheeger constant at least $i$ almost surely.
\end{prop}

We now produce an upper bound on the largest island simple random walk on
$H$ typically visits within $n$ steps.

\begin{lemma}\label{lem:largest_island}
  The probability that the random walk visits an $i(\alpha)$-island $I$
  with $|I|\geq m$ before time $n$ is at most $C (n e^{-cm} + e^{-cn})$. In
  particular, almost surely, the largest $i(\alpha)$-island visited
  within $n$ steps has size at most $C\log n$ for large enough $n$.
\end{lemma}

\begin{proof}
  Let $\{X_n\}_{n \ge 1}$ denote the random walk and let $P_n \subset H$ be the sub-map which is the hull of the faces
  incident to $(X_i)_{i\leq n}$. We define the exposed boundary of $P_n$ to
  be the set of vertices it shares with $H \setminus P_n$. Note that if
  $X_n$ is not in the exposed boundary of $P_{n-1}$, then $P_n=P_{n-1}$,
  whereas if $X_n$ is in the exposed boundary of $P_{n-1}$ then $P_n$ is
  constructed by adding to $P_{n-1}$ all neighbours of $X_n$ and any finite
  regions enclosed. This addition involves a geometrically distributed
  number of peeling steps at edges containing $X_n$. This associates to the
  random walk a sequence of peeling steps (see \cref{sec:peeling}). The number of peeling steps used
  to reveal $P_n$ is dominated by a sum of $n$ geometric variables, and so
  for some $a,c>0$, probability that more than $an$ peeling steps occur is
  at most $e^{-cn}$.

  Consider now the event $\cE_k$ that there is some $i$-isolated set $I$ of
  size at least $m$ such $I$ is disjoint of $P_{k-1}$ but not of $P_k$.  We
  argue that $\P(\cE_k|P_{k-1},X_k) \leq e^{-cm}$.  It then follows that the
  probability that $P_n$ intersects any large island is at most $an e^{-cm}
  + e^{-cn}$.

  We split according to the type of peeling step.  If the peeling step is
  of type $\alpha$, then $\cE_k$ can only occur if the $i$-isolated set
  intersects the only vertex in $P_k\setminus P_{k-1}$.  By
  \cref{lem:i-isolated} the probability of this event is exponentially
  small.

  The second possibility is that the peeling step connects an edge to some
  other vertex $v$, and that $v$ is already in the exposed boundary of
  $P_{k-1}$. In that case, the only way for $\cE_k$ to occur is if the
  isolated set $I$ is wholly contained in the Boltzmann triangulation
  surrounded by $P_{k-1}$ and the new face. Since the entire Boltzmann
  triangulation has exponentially decaying size (\cref{lem:1ball}), the probability of his
  event is also at most $e^{-cm}$.

  Finally, it is possible that the peeling step connects to some vertex $v$
  on the boundary of $H$ but not in $P_{k-1}$. In that case, the set $I$
  may be wholly in the Boltzmann triangulation (unlikely, as above) or may
  include $v$.  We split further, according to the distance of $v$ from
  $X_k$. The probability that $v$ is at least $m$ boundary edges of
  $H\setminus P_{k-1}$ away from $X_k$ is at most $e^{-cm}$. For each of
  the $2m$ vertices $v$ at distance at most $m$, the probability that $v$
  is in some large $i$-isolated set $I$ is at most $e^{-cm}$.  Thus the
  probability of connecting to some vertex $v$ which is in some large
  $i$-isolated $I$ is at most $(2m+1)e^{-cm} < e^{-c'm}$.
\end{proof}

\begin{lemma}\label{lem:hitting_bound}
  Let $G$ be a connected graph with $k$ edges, $S$ a non-empty subset and
  $\tau_S$ the hitting time of $S$. Then for any $m\ge 1$ and any vertex $x$
  \[
  \P_x(\tau_S > 4mk^2) \leq 2^{-m}.
  \]
\end{lemma}

\begin{proof}
  For $m=0$ the result is trivial. Using the commute time identity
  (\cite[Proposition 10.16]{MCMT}), for any vertex $y \notin S$, the
  expected hitting time of $S$ from $y$ is at most $2k^2$. By Markov's
  inequality the probability that $\tau_S>4k^2$ is at most $1/2$.  The
  result now follows by using induction on $m$ and the Markov property.
\end{proof}

\begin{prop}\label{lem:weak_speed}
  Almost surely,
  \[
  \liminf_{n \to \infty} \frac{d^{H}(X_n,\rho)\log^3n}{n} > 0.
  \]
\end{prop}

\begin{proof}
  Throughout this proof, fix $i=i(\alpha)$ defined in
  \cref{lem:i-isolated}, and consider the decomposition of $H$ into
  $i$-islands and $i$-oceans and the weighted graph $H_i$ from
  \cref{sec:anchored}. Let $(\tau_k)$ be the sequence of times
  when $X_t$ is in an $i$-ocean.  Using \cref{L:vol,lem:island_walk}, we
  conclude
  \begin{equation}
    \liminf_k \frac{d^{H}(X_{\tau_k},\rho)}{k} > 0. \label{eq:weak_speed3}
  \end{equation}
  
  Let $I_n$ be the size of the largest $i$-island visited by the random
  walker within $n$ steps.  Let $A_1$ be the event that for large enough
  $n$ we have $I_n<C\log n$.  For some $C>0$, \cref{lem:largest_island}
  ensures that a.s.\ $A_1$ holds, and we restrict to $A_1$ from here on.
  \cref{lem:hitting_bound} with $m=C'\log n$ shows that on $A_1$, for any
  $k<n$ we have $\tau_{k+1} > \tau_k + C^2 \log^3 n$ with probability at
  most $n^{-3}$ for large enough $C$ and $n$.  Let $A_2$ be the event that for
  large enough $n$, for all $k<n$ we have $\tau_{k+1} \leq \tau_k + C^2
  \log^3 n$.  By Borel-Cantelli, on $A_2$ holds a.s.\ on $A_1$.  On $A_2$
  we have $\tau_n\leq C^2 n \log^3 n$ for large enough $n$, and so $\log
  \tau_n\sim\log n$ and
  \[
  \liminf_n \frac{d^{H}(X_{\tau_n},\rho) \log^3 n}{\tau_n} > 0.
  \]

  Furthermore, on $A_1$ for $t\in[\tau_k,\tau_{k+1}]$ we have
  $d(X_t,X_{\tau_k}) \leq C\log n$ for large enough $n$, which allows us to
  interpolate and the claim follows.
\end{proof}

Now we recall a result due to Carne and Varopoulos.

\begin{thm}(\cite{C85,V85})\label{thm:CV}
  Let $X_n$ be a simple random walk on a graph $G$ with spectral radius
  $\rho$. For any two vertices $x,y$ in $G$ 
  \[
  \P_x(X_n=y) \leq 2 \rho^n \sqrt{\frac{\deg(y)}{\deg(x)}}
  \exp\left(- \frac{d^G(x,y)^2}{2n} \right).
  \]
\end{thm}

\begin{lemma}\label{lem:finite_hit_bdry}
  Almost surely, the random walk on $H$ visits $\partial H$ only finitely often.
\end{lemma}

\begin{proof}
  \cref{cor:ball_boundary} implies that for some $C$, $\H$-almost
  surely, $|B_n\cap\partial H| \leq Cn$ for all large enough $n$. Assume
  this holds, so there are at most $Cn$ possible values of $X_n$ that we must
  eliminate.  \cref{lem:weak_speed} implies that a.s.\ for all large enough
  $n$ we have $d^H(X_n,\rho) > n^{2/3}$.  For any $y\in\partial H_n$ with
  $d^H(y,\rho)\in[n^{2/3},n]$ we have from \cref{thm:CV} that $P_H(X_n=y)
  \leq 2 \deg(y) e^{-n^{1/3}/2}$. 
  A union bound over $y\in\partial H$ not too close to $\rho$ along with appeals to \cref{cor:ball_boundary} and the Borel-Cantelli lemma completes the proof.
\end{proof}

\begin{lemma}\label{lem:pos_speed}
  Let $s>0$ be as in \cref{lem:pos_prob}. Almost surely,
  \[
  \liminf_{n\to \infty}\frac{d^{H}(X_n,\rho)}n > s.
  \]
\end{lemma}

\begin{proof}
  For $v\in\partial H$, let $\cE_{v,k}$ be the event that $X_k=v$ and the
  walk never visits the boundary thereafter.  \cref{lem:finite_hit_bdry}
  ensures that $P_H\left( \cup_{v,k} \cE_{v,k} \right) = 1$.  From
  \cref{lem:pos_prob}, the Markov property of random walk on $H$ and
  translation invariance of $H$ we deduce that on each $\cE_{v,k}$, almost
  surely $\liminf d(X_n,\rho)/n \geq s$.
\end{proof}

Now we turn to prove \cref{T:speed}. Let $B(\partial H,r)$ denote the hull of the ball of radius $r$ around $\partial H$.
\begin{lemma} \label{lem:vol_strip} For all $\eps>0$, there exists a
  $\delta>0$ depending only upon $\eps,\alpha$ such that almost surely for
  all large enough $n$
  \[
  |B_n \cap B(\partial H,\delta n)| < \exp(\eps n).
  \]
\end{lemma}

\begin{proof}
  Follows from \cref{L:vol,prop:dist_along_bdry} and translation invariance.
\end{proof}

\begin{proof}[Proof of \cref{T:speed}]
  Fix $\eps=s^2/9$, where $s$ is as in \cref{lem:pos_prob}.  Choose
  $\delta$ such that \cref{lem:vol_strip} is satisfied, and let $A_n = B_n
  \cap B(\partial H,\delta n)$. Now consider the event
  \[
  \cE_n =\left \{|A_n| \leq \exp(\eps n)\right\}
  \cap \left\{d^H(X_n,\rho) \geq s n/2 \right \}.
  \]
  Notice $\cE_n$ occurs almost surely for all large enough $n$ (from
  \cref{lem:vol_strip,lem:pos_speed}).  Now using \cref{thm:CV}, we obtain 
  \begin{align*}
    P_H(X_n \in A_n, \cE_n)
    &\leq \sum_{\stackrel{y \in A_n}{d^H(y,\rho) \geq sn/2}}
    2\deg(y) \exp(-s^2 n/8)\\
    &= 2|A_n| e^{-s^2 n/8} \leq 2 e^{-s^2 n/72}.
  \end{align*}
  The Borel-Cantelli lemma implies that the events $\{X_n \in A_n\}$ occur
  finitely often almost surely. This completes the proof of the Theorem.
\end{proof}


\section{Return probabilities}\label{sec:4}
In this section we prove \cref{T:return_prob}.

\subsection{Upper bound}

We first get an annealed version of the upper bound on the return
probability.

\begin{lemma}\label{lem:annealed}
  There exists a constant $c>0$ such that
  \[
  \P(X_n = \rho) \le e^{-cn^{1/3}}.
  \]
\end{lemma}

\begin{proof}
  Let $i=i(\alpha)$ be as in \cref{lem:i-isolated}, so that large
  $i$-islands are exponentially uncommon. We prove the claim for $n$ such that
  $n^{-1/3}\le i$.  By changing $c$ it holds for smaller $n$ as
  well. Let $\eps = n^{-1/3}$, and consider the weighted graph $H_\eps$
  constructed in \cref{sec:anchored}. Note that there is a natural
  coupling of the random walk on $H$ and on $H_\eps$ so that the two agree
  until the first time that the random walk on $H$ visits an $\eps$-island.

  Let $B_n$ be the event that the simple random walk visits an $i$-island
  of size at least $n^{1/3}$ within $n$
  steps.
  \begin{equation}
    \P\big(X_n = \rho\big) \le \P\big(\{X_n = \rho\} \cap B_n^c\big) +
    \P\big(B_n\big).
    \label{eq:break}
  \end{equation}
  By \cref{lem:largest_island}, we conclude
  \begin{equation}
    \P(B_n) \le Cn \exp(-cn^{1/3})  \label{eq:annealed1}
  \end{equation}
  for some $c,C>0$.

  \cref{prop:sinking} implies that all the $i$-islands of size at most
  $n^{1/3}$ are in $H_\eps$. Thus on the event $B_n^c$, the random walks on
  $H$ and $H_\eps$ coincide at least up to time $n$. Hence
  \begin{equation}
    \label{eq:Prrbound}
  \bP_H(\{X_n=\rho\} \cap B_n^c) \leq \bP_{H_\eps}(X_n=\rho).
  \end{equation}
  By \cref{prop:cheeger_weighted}, $H_\eps$ has $\eps$-expansion, and so by
  Cheeger's inequality, the spectral radius of the random walk operator on
  $H_\eps$ is at most $1 - \eps^2/2$. It follows that
  \[
  \bP_{H_\eps}(X_n=\rho) \leq \left(1 - \frac{\eps^2}{2} \right)^n \leq
  \exp(-n^{1/3}/2),
  \]
  hence the lemma follows by combining this with
  \cref{eq:break,eq:Prrbound,eq:annealed1}. 
\end{proof}

\begin{proof}[Proof of \cref{T:return_prob} upper bound]
  This follows from \cref{lem:annealed} together with Markov's inequality
  and the Borel-Cantelli lemma.
\end{proof}

\subsection{Lower bound: existence of traps}

To prove the lower bound, we show that the simple random walk spends much
time in certain traps with not too small probability.  The argument then
consists of three parts. First, once a trap is reached, there is some
probability of staying inside it much of the time. Second, sufficiently
large traps exist reasonably close to the root. Finally, the probability of
reaching the trap, and returning from it to the root are not too small. 

The traps we shall consider resemble long paths. Let us start with a lemma
about the simple random walk on $\Z$ which shows that with exponential
cost, the walker can stay within the interval $(0,n)$ for $n^3$ steps.

\begin{lemma}\label{SRWestimate}
  Consider a random walk $\{X_i\}_{i\ge0}$ on $\Z$ with steps uniform in
  $\{-1,0,1\}$, starting from $1$. For all $t \ge n^3/2$, there exists
  $c>0$ such that for all large enough $n$,
  \[
  \P(X_t=1, X_1,\dots, X_t \in [1,n] ) \ge e^{-ct/n^2}.
  \]
\end{lemma}

(The assumption on $t$ can easily be relaxed to $n^2\log n$, which we do
not need.)

\begin{proof}
  The probability that the random walk reaches $\lf n/2\rf$ before reaching
  $0$ is $2/n$. For any $k\in[n/4,3n/4]$, the probability that the random
  walk started at $k$ does not exit $(0,n)$ for $n^2$ steps, and after
  $n^2$ steps is again in $[n/4,3n/4]$ is at least some $c>0$. Using the
  Markov property and iterating this event $\lfloor (t-n)/n^2\rfloor$
  times, we get that the random walker stays in $[1,n]$ for $t-n$ steps and
  ends up in $[n/4,3n/4]$ is at least $\exp(-ct/n^2)$. Finally. the
  probability that the walker reaches $1$ in the next $n$ steps is at least
  $\exp(-cn)$. But since $t>n^3/2$, we have the desired result.
\end{proof}

Let us now define our traps. A trap of order $n$ consists of $n+1$
triangles with disjoint vertices, each inside the previous one (ordered and
numbered $0,\dots,n$), with edges connecting consecutive triangles as shown
in \cref{fig:trap} and no other vertices between triangles or within the
last triangle.  Only vertices of triangle $0$ are connected to the rest of
the map. When the random walk is at a vertex of the $k$th triangle for
$0<k<n$, it moves to a vertex in triangle $k'$ which is equally likely to
be each of $\{k-1,k,k+1\}$.

\begin{figure}
  \centering{\includegraphics[width=.5 \textwidth]{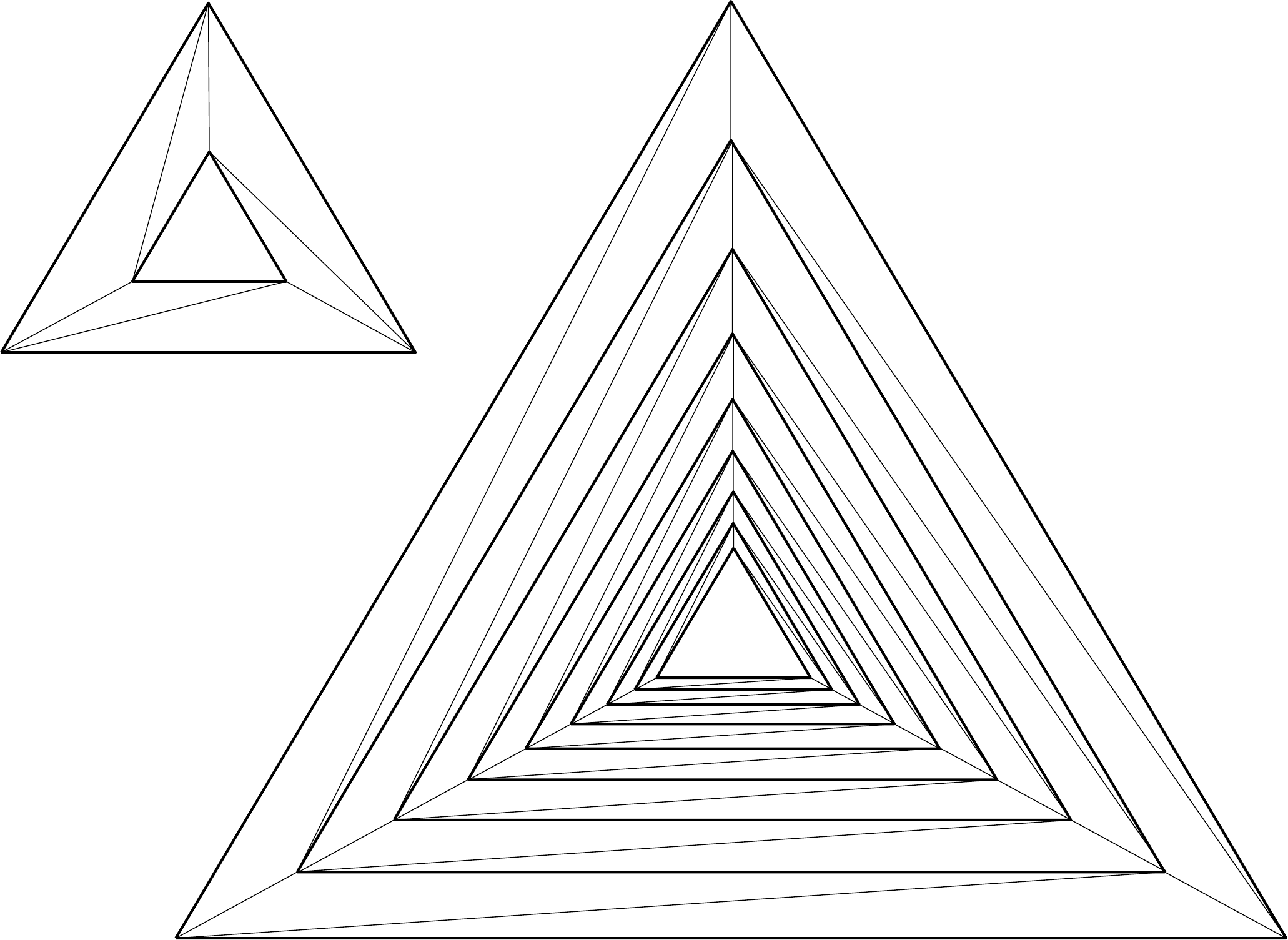}}
  \caption{Traps of order $1$ (left) and $8$ (right).} \label{fig:trap}
\end{figure}

\begin{corollary}\label{cor:trap_delay}
  A simple random walk started from a vertex of triangle $1$ of a trap of
  order $n$ has probability at least $e^{-ct/n^2}$ of being back at
  triangle $1$ at time $t$, for any $t>n^3/2$.
\end{corollary}

Now perform peeling to reveal the hulls of the ball of radius $r$ around
the root vertex as described in \cref{sec:peeling}. Consider the event
$\cE_n$ that in this process, a step of the form $(R,2)$ occurs and the
finite triangulation in the area enclosed by the revealed triangle is a
trap of order $n$. The number of steps is exponential in $r$, and the
probability of finding a trap is exponential in $n$. This suggests that
traps can be found with high probability.

\begin{lemma}\label{lem:revealing_trap}
  There exists a positive constant $C$ depending only on $\alpha$ such that
  for all $n\ge 1$
  \[
  \H\left(\cE_n \text{ occurs before revealing $B_{\lfloor Cn
        \rfloor}$}\right) \xrightarrow[n\to\infty]{} 1.
  \]
\end{lemma}

\begin{proof}
  \cref{L:vol} shows that $|B_n|$ is at least $e^{cn}$ with exponentially
  high probability for some small enough $c>0$. Now recall that the
  increments in the volume of the revealed triangulation in the peeling
  steps are i.i.d. with finite expectation. Hence an application of
  Markov's inequality shows that the number of steps needed to reveal the
  hull of radius $n$ is at least $\exp(cn)$ with probability at least
  $1-\exp(-c'n)$.

  The domain Markov property and \cref{lem:prob} imply that the number of
  steps needed until $\mathcal E_n$ occurs is a geometric variable with
  probability of success at least $\exp(-c''n)$ for some constant $c''>0$
  depending only on $\alpha$. The lemma now follows by choosing $C$ large
  enough depending upon $c''$.
\end{proof}

\begin{corollary}\label{cor:trap_far_near}
  With exponentially high probability there exists a trap of order $n$ at
  distance at least $n$ and at most $Cn$ from the root for large
  enough $C$.
\end{corollary}
\begin{proof}
Condition on $B_n$. Now root $H \setminus B_n$ on an edge in the
exposed boundary of $B_n$ and appeal to domain Markov property and \cref{lem:revealing_trap}. 
\end{proof}

\subsection{Lower bound: getting to a trap}

We still need to show that the probability of reaching a trap at a distance
$\ell$ from the root is at least $\exp(-c\ell)$ for some $c>0$.  We can
estimate the probability of the simple random walk reaching the trap by
moving along a given geodesic joining the root and the trap.  If the
degrees of the vertices along such a path are $d_0,\dots,d_{\ell-1}$ then
the probability of following the path is
\begin{equation}
  \prod d_i^{-1} \geq \left(\frac{\ell}{d_0+\dots+d_{\ell-1}}\right)^\ell
  \label{eq:AMGM} 
\end{equation}
(by the A-G mean inequality). For this reason we prove the following lemma
about average degrees along paths in $H$.  Call a simple path in $H$ {\bf
  $\gamma$-bad} if the average of the degrees of vertices along the path is greater
than $\gamma$.

\begin{lemma}\label{lem:gamma_bad}
  There exists a constant $\gamma>0$ depending only on $\alpha$ such that for
  all $n$ the probability that there exists a $\gamma$-bad path of length
  $n$ in $H$ starting from $\rho$, avoiding $\partial H$ except at $\rho$
  is at most $e^{-cn}$.
\end{lemma}

Before proving this let us introduce some notations. For any $n \ge 1$ and
given an instance of $H$, let $\cP_n(H)$ denote the set of simple paths of
length $n$ in $H$ starting from $\rho$ and avoiding $\partial H$ except at
$\rho$.  Let $\H^{(n)}$ be the measure defined by its Radon-Nykodim
derivative given by $\frac{d\H^{(n)}}{d\H} = \# \cP_n$. Note that $\H^{(n)}$ is not a
  probability measure, but has total mass $\E\left(\#\cP_n\right)$.

Let $\wH^{(n)}$ be the measure of the pair $(H_n,P_n)$ where $H$ has law
$\H^{(n)}$ and $P_n$ is a uniformly picked path from $\cP_n$. Given a pair
$(H_n,P_n)$ let $\Cut(H_n,P_n)$ be the map obtained by cutting $H_n$ along
$P_n$ as shown in \cref{fig:zip}. Observe that since $P_n$ is a simple path
avoiding the boundary, $\Cut(H_n,P_n)$ is a half planar triangulation.
Note that not every map may result from this procedure: since $H_n$ has no
self-loops, $\Cut(H_n,P_n)$ cannot have an edge between boundary vertices
that come from the same vertex of $P_n$.

\begin{figure}[t]
  \centering{\includegraphics[width=0.75 \textwidth]{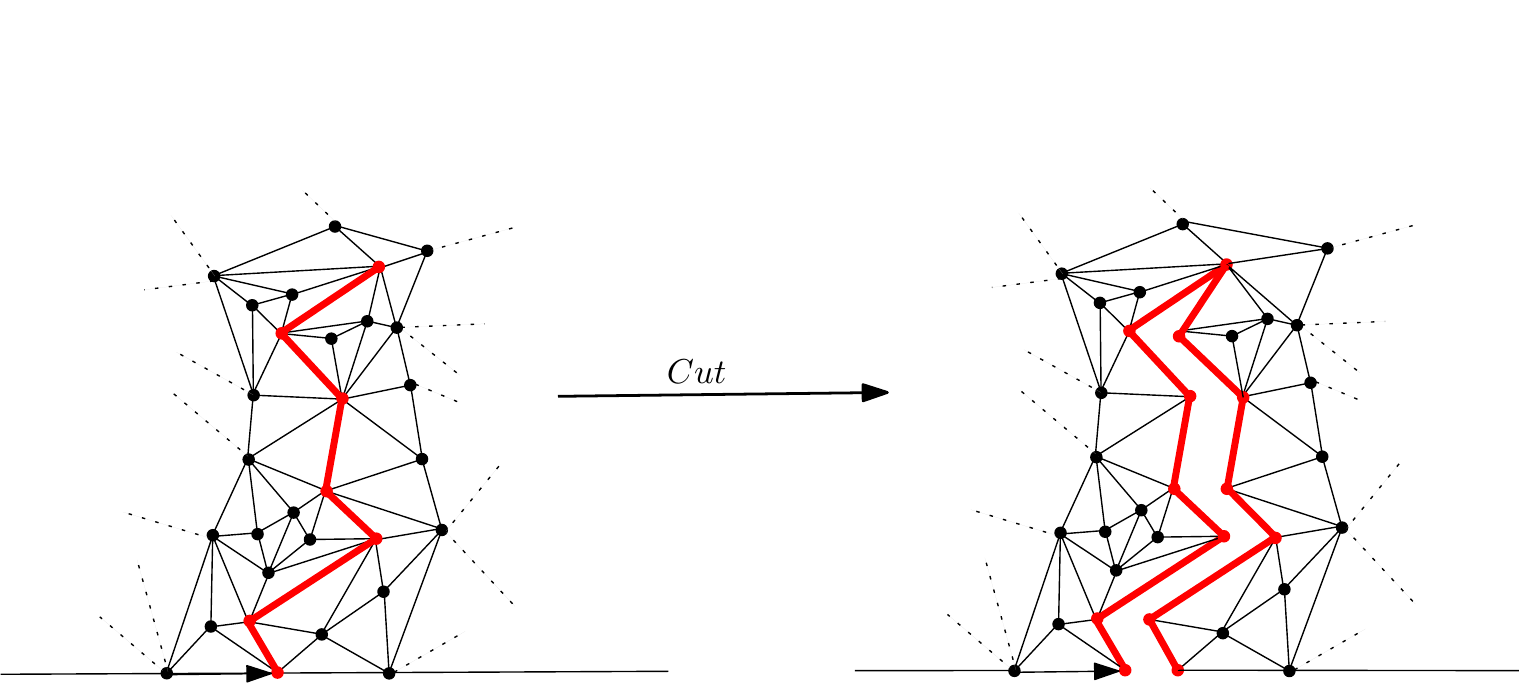}}
  \caption{The operation $\Cut$ takes a map with a simple path avoiding the
    boundary (in red) and produces a half plane map by cutting along the
    path}
  \label{fig:zip}
\end{figure}

A direct consequence of \cref{lem:prob} is that
\begin{equation}
  \wH^{(n)} = \left (\frac\alpha\beta\right)^n \H \circ \Cut.
  \label{eq:relation}
\end{equation}
This is so since the probability of any simple event $Q\subset H$ is in
agreement.  The factor $(\alpha/\beta)^n$ above appears because $\Cut$
turns $n$ internal vertices of the map into boundary vertices.

\begin{proof}[Proof of \cref{lem:gamma_bad}]
  The expected number of $\gamma$-bad paths of length $n$ is given by
  $\wH^{(n)}\left(P_n \text{ is $\gamma$-bad }\right)$. Thus it suffices to
  prove the exponential bound on this quantity.  Now, the sum of the
  degrees of the vertices in $P_n$ is $2n$ less than the sum of the degrees
  of the vertices in a segment of length $2n-1$ along the boundary of
  $\Cut(H_n,P_n)$, just to the right of the root.  The lemma and the choice
  of $\gamma$ follow by \eqref{eq:relation} and \cref{lem:hull1}, taking
  $b>\log(\alpha/\beta)/2$.
\end{proof}

For positive constants $\gamma,C_1,C_2$ consider the event $\A_n$:
\begin{itemize}
\item There exists a trap of order $n$ at distance at least $n$ and at most
  $C_1 n$ from the root,
\item $| B_{\lfloor C_1n \rfloor}\cap \partial H|\le C_2n$.
\item There does not exist a $\gamma$-bad simple path of length at least
  $n$ starting from any boundary vertex $v$ in $B_{\lfloor C_1n \rfloor}$
  avoiding $\partial H$ except at $v$.

\end{itemize}

\begin{lemma}\label{lem:A_n_prob}
  There exist choices of positive constants $\gamma,C_1,C_2$ depending only
  on $\alpha$ such that $\H(\A_n) \geq 1-e^{-cn}$ for some $c>0$.
\end{lemma}

\begin{proof}
Exponential bound on the complement of the first event follows from \cref{cor:trap_far_near}.
The exponential bound on the complement on second event above follows from \cref{cor:ball_boundary}.
 On the second event, the bound on the third event follows from \cref{lem:gamma_bad}, translation invariance and union bound.
\end{proof}

\begin{lemma}\label{lem:A_n}
  On the event $\A_n$, for $t\in[n^3,(n+1)^3]$ we have 
  \[
  \bP_H \left(X_t = \rho \right) \ge e^{-cn}
  \]
  for some $c>0$.
\end{lemma}

\begin{proof}
  On the event $\mathcal A_n$, consider a geodesic path joining the root
  and a vertex $v$ in triangle $1$ of the promised trap.  Suppose the
  geodesic path hits $\partial H$ for the last time at $u$. Let $\Gamma_1$
  be the segment of the boundary from $\rho$ to $u$, and $\Gamma_2$ the
  segment of the geodesic from $u$ to the trap. Let $\Gamma$ be their
  concatenation, and let $\ell$ denote the length of $\Gamma$. We wish to
  show that with probability at least $\exp(-cn)$ the walker moves along
  $\Gamma$ to $v$, spends $t-2\ell$ steps in the trap such that at the end
  of these steps it is again at $v$, and then returns to the root along
  $\Gamma$.

  Notice that on the event $\A_n$ we have $|\Gamma_1|\leq C_2 n$ and
  $|\Gamma_2| \leq \gamma\cdot C_1 n$, since the length of $\Gamma_2$ is at
  most $C_1 n$. Thus $|\Gamma| \leq (\gamma C_1+C_2)n$.  Also, $n\leq \ell
  \leq (C_1+C_2)n$, since the trap is outside $B_n$ and by the choice of
  $\Gamma$. Thus the average degree along $\Gamma$ is at most $\gamma
  C_1+C_2$, and the probability that the simple random walk moves from the
  root to the $v$ along $\Gamma$ is or similarly comes back from the vertex
  $v$ to the root along $\Gamma$ has probability at least $(\gamma
  C_1+C_2)^{-(C_1+C_2)n}$. The probability of returning along $\Gamma$ is
  the same. Finally \cref{cor:trap_delay} (by plugging in $t-2\ell$)
  ensures that with probability at least $\exp(-cn)$, the random walker
  stays inside the trap for $t-2\ell$ steps such that at the end of these
  steps it is in $v$.
\end{proof}

\begin{proof}[Proof of \cref{T:return_prob} lower bound]
  Follows from \cref{lem:A_n,lem:A_n_prob} and an application of
  Borel-Cantelli lemma.
\end{proof}

\section{Comments and Open questions}\label{sec:open}

\paragraph*{Existence of the speed.} 
First, since the ergodic theorem does not apply directly, it is not obvious
whether the speed actually exists, as it does in the full plane map.  Is
there an almost sure constant speed for the simple random walk in $H$?  In
other words, does the sequence $d(X_n,\rho)/n$ converge almost surely to
some constant? In a similar token: does the speed away from the boundary
exist and is constant?  One approach would be to study renewal times, at
which the walker reaches a new distance from the boundary, and never
descends below that distance subsequently.  If these renewal times exist
and are sufficiently small, it would follow that the distance from the
boundary has an almost sure asymptotic speed and even that it has Brownian
fluctuations.  The distance from the root seems less accessible.

\paragraph*{Dependence of $\alpha$.}
If the speed $s_\alpha>0$ exists for simple random walk on a map with law
$\H_\alpha$ is it increasing in $\alpha$? Does $s_\alpha \to 0$ as $\alpha
\to 2/3$ and $s_\alpha \to 1$ as $\alpha \to 1$? Is the speed continuous in
$\alpha$? These questions are open also for the full plane hyperbolic maps,
for which it is known that the speed exists.

\paragraph*{Coupling as $\alpha\to 2/3$.}
Is the half plane UIPT recurrent? It is shown in \cref{lem:coupling} that
we can obtain a coupling such that $H$ is almost surely a sub-map of its
full plane version $F$.  Does there exist a similar coupling between
half-plane and full-plane UIPT?  This would imply in particular that the
half plane UIPT is recurrent since the UIPT is recurrent \cite{GN12}.
Existing couplings have the property that the distance of $H$ from the root
of $F$ tends to infinity as $\alpha\to 2/3$.  Can this be avoided?  One
approach which is known to fail is to take two copies of the half-plane
UIPTs and glue them along the boundary (Nicolas Curien, personal
communication).

\bibliographystyle{abbrv}
\bibliography{rwsuper}

\noindent {\sc Omer Angel, Asaf Nachmias, Gourab Ray} \\
Dept.\ of Mathematics, UBC. \\
{\tt angel@math.ubc.ca, asafnach@gmail.com, gourab1987@gmail.com}

\end{document}